\documentclass[11pt]{article}
\usepackage{geometry}
\usepackage[utf8]{inputenc}
\usepackage{amsmath,amsthm,amssymb,color,graphicx,url,hyperref,diagbox,enumerate,cite,tikz,authblk}
\usepackage{appendix}

\newtheorem{theorem}{Theorem}

\newtheorem{problem}[theorem]{Problem}

\newtheorem{lemma}[theorem]{Lemma}

\newcounter{encoding}

\newcommand{\N}{\mathbb{Z}^+}

\newcommand{\Z}{\mathbb{Z}}

\newcommand{\lengthconstant}{19}

\newcommand{\floor}[1]{\left\lfloor #1 \right\rfloor}

\title{The Fibonacci numbers are not 3-accessible}
\author{William J. Wesley\thanks{Discrete Mathematics Group, Institute for Basic Science (IBS), Daejeon,
South Korea. \\This work was supported by the Institute for Basic Science (IBS-R029-C1).}}
\date{\today}
\affil{}

\begin{document}

\maketitle

\section{Introduction}

A \emph{$D$-diffsequence} is a sequence of integers $x_1 < \dots < x_k$ such that $x_{i+1} -x_i \in D$ for $1 \le i \le k-1$. The set $D$ is called \emph{$r$-accessible} if every $r$-coloring of the positive integers contains arbitrarily long monochromatic $D$-diffsequences. The maximum $r$ for which $D$ is $r$-accessible is called the \emph{degree of accessibility} of $D$ and denoted $\operatorname{doa}(D)$. 

Diffsequences have been studied for many sets $D$ in several recent papers \cite{CliftonDiffsequences,RamseyFunctions_RestrictedGaps,Fibonacci_Ramsey,WJWFib4,LandmanRobertson,TalwarGupta,LandmanRobertsonRobertson}. We refer the reader to \cite{LandmanRobertson}, Section 10.3 for an overview of the topic. The set of primes was recently determined to be 2-accessible \cite{Quester2026primes2accessible}. The degree of accessibility of the set $F = \{1,2,3,5,8,\dots\}$ of Fibonacci numbers was first studied in 2007 by Landman and Robertson in
\cite{LandmanRobertson_SpecialGaps}, where the authors showed $\operatorname{doa}(F) \ge 2$. They left determining the exact value of $\operatorname{doa}(F)$ as an open problem.
The first upper bound was given in
\cite{Fibonacci_Ramsey}, where the authors gave a 6-coloring of the positive integers that does not contain any 2-term monochromatic $F$-diffsequences, which gives $\operatorname{doa}(F) \le 5$. Computational data from \cite{Fibonacci_Ramsey} strongly suggested that $\operatorname{doa}(F) \le 3$, and the author proved this in \cite{WJWFib4}. Quester later gave a different proof in \cite{QuesterDiffsequences}.  However, for 3 colors, the picture was unclear, and the question of whether the Fibonacci numbers are 3-accessible also appeared as an open problem in several works \cite{LandmanRobertson,QuesterDiffsequences,OpenProblemGarden}.

A natural class of $r$-colorings is those of the form $\chi_{\alpha}$, where $\chi_{\alpha}$ assigns $n$ the color $i$ if $\{\alpha n\} \in [\frac{i-1}{r}, \frac{i}{r})$, where $\{x\} = x - \lfloor x \rfloor$ denotes the fractional part of $x$. Such a coloring is called an equal-interval \emph{coding} of the rotation $\alpha$. These colorings (and related ones) have been studied from the perspective of ergodic theory and combinatorial words, and they are especially useful for diffsequences. One can show that if there exists a real number $\alpha$ such that \begin{equation} \label{EqAlpha} \frac{1}{k-1}\le \{\alpha d\} \le 1-\frac{1}{r}
\end{equation}for all $d \in D$, then $\chi_\alpha$ contains no monochromatic $k$-term $D$-diffsequences, and in particular, $D$ is not $r$-accessible. Clifton used this idea to prove that certain \emph{dividing sequences} are not 2-accessible \cite{CliftonDiffsequences}. Quester went on to show that if $D = \{d_i\}$ satisfies $d_{i+1}/d_i  > 2+ \frac{1}{r-1} +\delta$ for some $\delta >0$ and all $i$, then $D$ is not $r$-accessible \cite{QuesterDiffsequences}. 

Unfortunately, there is no $\alpha$ that satisfies \eqref{EqAlpha} for $d \in F$ for $r = 3$ colors (see the concluding remarks of \cite{QuesterDiffsequences}). However, the recurrence properties of the Fibonacci numbers allow for some control over $\{\alpha d\}$ for $d \in F$. In particular, we can find an $\alpha$ such that the sequence of values $\{ \alpha f_n\}$ is ``almost" periodic (more precisely, this sequence converges to a periodic sequence of rational numbers). Considering a larger class of colorings than the equal interval coding will allow us to find a 3-coloring of $\N$ that avoids long monochromatic $F$-diffsequences. This note proves the titular result, settling the question of the degree of accessibility of the Fibonacci numbers. 

\begin{theorem}\label{MainThm}
    There exists a 3-coloring of $\N$ that does not contain monochromatic $\lengthconstant$-term $F$-diffsequences. In particular, $\operatorname{doa}(F) = 2$. 
\end{theorem}

\section{Proof of Theorem \ref{MainThm}}
\label{SectionMain}

Throughout, we let $m = 232$, $\phi = \frac{1+ \sqrt{5}}{2}, \psi = \frac{1-\sqrt{5}}{2}$, $\alpha = 184 + 8\phi$, and take note of the fact that $|\psi| < 1$. We let $f_0 = 0, f_1 = 1$, and $f_n = f_{n-1} + f_{n-2}$ for $n \ge 2$, so that $F = \{f_n: n \ge 1\}$. 

The first result we need is a simple lemma on the recurrence properties of the Fibonacci numbers. 

\begin{lemma}\label{LemmaRecur}
For integers $a,b$, we have $$(a+b\phi)f_n = af_n + bf_{n+1} - b \psi^n$$ for all $n\ge 0$.
\end{lemma}
\begin{proof}
    Let $L_n = (a+b\phi)f_n$ and $R_n = af_n + bf_{n+1} - b \psi^n.$
    A straightforward computation shows $L_n = R_n$ for $n = 0,1$. It is also not hard to see that $L_n$ satisfies the Fibonacci recurrence $L_{n} = L_{n-1} + L_{n-2}$ for $n \ge 2$, and using the fact that $\psi^n = \psi^{n-1} + \psi^{n-2}$ for all $n$, we see that $R_n = R_{n-1} + R_{n-2}$ for all $n\ge 2$ as well. Therefore both sides of the equation agree for all $n \ge 0$. 
\end{proof}

For any positive integer $n$, let $\rho(n) = \floor{\alpha n} \pmod{m} \in \Z_m$. Define a 3-coloring $C$ on $\Z_m$ whose  color classes $C_i = \{ x \in \Z_m : C(x) = i\}$, $1\le i \le 3$, are 

\begin{align*}
C_1&=[0,24]\cup\{68\}\cup[87,98]\cup[125,126]
     \cup[148,155]\cup[187,196]\cup[200,217]\cup\{231\}, \\
     C_2 &=[41,55]\cup[117,124]\cup[127,143]\cup[156,186]\cup[218,224],\\
     C_3 &= [25,40]\cup[56,67]\cup[69,86]\cup[99,116]\cup[144,147]\cup[197,199]\cup[225,230].
\end{align*}

We will show that the 3-coloring $\chi: \N \to \{1,2,3\}$ given by 
$$\chi(n) = C(\rho(n))$$ does not contain any long monochromatic $F$-diffsequences. 

The key to the proof is that there are not too many possibilities for the change in $\rho$ when translating the argument by a Fibonacci number. Let \begin{multline*}S = \{ 7, 8, 9, 55, 56, 57, 126, 127, 128, 129, 159, 160, 161, 162,\\ 183, 184, 185, 191, 192, 193, 196, 197, 199, 200, 201 \}.\end{multline*}

\begin{lemma} \label{LemmaS}
    For all $f \in F$, we have $\rho(x+f) -\rho(x) \in S$. 
\end{lemma}
\begin{proof}
    For all real $x,y$, we have $\floor{x+y} - \floor{x} \in \{\floor{y},\lceil y\rceil\}$. Then the set of possible values of $\rho(x+f) - \rho(x)$ is a subset of $\{ \floor {\alpha f} \pmod {m},\lceil \alpha f\rceil \pmod{m}: f \in F\}$.

     From Lemma \ref{LemmaRecur}, we have $\alpha f_n = 184 f_n + 8f_{n+1} - 8\psi^n$. It is straightforward to check that the sequence $a_n = 184 f_n + 8f_{n+1}$ is periodic modulo $m$ with period 7. Its orbit $\{a_n: n \ge 0\}$ is the set $\mathcal O :=\{8,56,128,160,184,192,200\}$. 

    For $1 \le n \le 4$, a direct calculation (modulo $m$) gives $\floor{\alpha f_n} \in \{126,161,196\}$ and $\lceil \alpha f_n \rceil \in\{127,162,197\}$.
    For $n \ge 5$, we have $|8\psi^n| < 1$, so $|a_n - \alpha f_n| < 1$, and the result follows since $$S = \{126,161,196\} \cup \{127,162,197\} \cup \mathcal O \cup (\mathcal O +1) \cup (\mathcal O -1).$$
\end{proof}

Let $G$ be the directed Cayley graph with vertex set $\Z_m$ where there is a directed edge $(x,y)$ if and only if $y-x \in S$. By Lemma \ref{LemmaS}, we have that $\chi$ does not contain an arbitrarily long monochromatic diffsequence if for all colors $i = 1,2,3$, the induced subgraph $G[C_i]$ has no directed cycles. This can be verified easily via a standard topological sorting algorithm. The longest paths in $G[C_i]$ have 19, 18, and 14 vertices for $i = 1,2,3,$ respectively, so the longest possible length of a monochromatic diffsequence is at most 19. 

\section{Discussion}\label{SectionDiscussion}
 We conclude by offering some brief observations on the proof in Section \ref{SectionMain} to give some insight into where the choices of $m, \alpha$, and $C$ came from. These were all specially chosen for the Fibonacci numbers $F$, but the heuristics and methods described here are applicable to other sequences given by linear recurrences, for instance the Lucas, Pell, and Perrin numbers. We are optimistic that the bounds on the degree of accessibility for all these numbers can be improved from those given by the author in \cite{WJWFib4}.   

We begin with the modulus $m$. We want to find a sequence that is periodic and satisfies the Fibonacci recurrence modulo $m$. If $t$ is a desired period, we need a nontrivial solution to the following system of equations modulo $m$: 

$$x_{i} +x_{i+1} = x_{i+2} : i \in \Z_t. $$ Writing this system as $M_t x = 0$, we can compute that $\det M_7 = 29$. Therefore for any $m$ divisible by 29, there is a nontrivial period 7 sequence of elements of $\Z_m$ that satisfies the Fibonacci recurrence modulo $m$. One such sequence modulo 29 has the repeating block $B =(1, 24, 25, 20, 16, 7, 23)$. 

Set $a = 23, b =1$, $g_n = (a+b\phi)f_n$, and $r_n = af_{n}+bf_{n+1}$. We have $r_0 = 1, r_1 = 24$, and $r_n = r_{n-1} + r_{n-2}$. Therefore the sequence $r_n$ (modulo 29) is precisely the repeating block $B$. Lemma 2 gives that $|g_n -r_n|$ is small for large $n$, so the values of $g_n$ modulo 29 are approximately the values in $B$. In particular, $g_n$ takes on a small set of values modulo 29. However, this set is still too large relative to the modulus 29. We can correct this by dilating by a factor of 8, giving our values $\alpha = 8(23+\phi), m = 8(29)$. This dilation does not enlarge the set $S$ of values of $\floor{\alpha f_n}$ modulo 29 too much, but allows for more flexibility in finding a 3-coloring $C$ of the directed Cayley graph $G$. 

The problem of finding $C$ is related to finding the \emph{dichromatic number} of a directed graph, which is the fewest number of colors needed to color the vertices such that each color class is acyclic \cite{NEUMANNLARA_DichromaticNumber}. For our directed graph $G$, we need to determine whether $G$ has dichromatic number at most 3. Unfortunately, deciding whether a directed graph has dichromatic number at most 2 is already NP-complete \cite{CircularChromaticNumberBokal_et_al}. In practice, such a coloring can be found by integer programming or SAT solving (as was done here), but it is not obvious how to find colorings generally. The author plans to refine the techniques to find $m, \alpha,$ and $C$ in future work.     

Finally, we mention that while this work establishes the exact value $\operatorname{doa}(F) = 2$, there is still more to the story. The value 19 for the number of terms of the longest monochromatic diffsequence was not optimized, and the author suspects that there is a coloring that avoids shorter monochromatic diffsequences. However, we know from \cite{Fibonacci_Ramsey} that 5-term monochromatic $F$-diffsequences are unavoidable. It would be interesting to find the smallest number of terms in a  monochromatic diffsequence that can be avoided with a 3-coloring. 

\section*{Acknowledgments}
The author thanks Alexander Clifton for many helpful discussions on diffsequences. 

\section*{AI declaration}
The author acknowledges that he found the proof in Section \ref{SectionMain} with assistance from an advanced LLM. The mathematical commentary in Section \ref{SectionDiscussion} is the author's own. All results were checked and independently verified by the author. The entirety of this note was written by the author alone without any AI assistance.  
\bibliographystyle{abbrv}
\bibliography{Bibliography.bib}

\end{document}